\documentclass[11pt]{article}
\usepackage{amsthm,amstext}
\usepackage{amsmath}
\usepackage{graphicx}
\usepackage{amsfonts}
\usepackage{amssymb}
\theoremstyle{plain}
\newtheorem{theorem}{Theorem}[section]

\newtheorem{proposition}[theorem]{Proposition}

\theoremstyle{definition}
\newtheorem{definition}[theorem]{Definition}
\newtheorem{corollary}[theorem]{Corollary}
\newtheorem{example}{\sc Example}
\theoremstyle{remark}
\newtheorem{remark}{\sc Remark}

\date{}
\title{\bf On $Q$-Fuzzy Ideal Extensions in Semigroups}\vspace{.25 in}
\author{ {\bf Samit Kumar Majumder}\\
Tarangapur N.K High School, Tarangapur\\
Uttar Dinajpur, West Bengal-733 129, INDIA\\
{\tt samitfuzzy@gmail.com}
 }

\begin{document}
\maketitle

\begin{abstract}

In this paper the concept of
extension of a $Q$-fuzzy ideal in semigroups has been introduced and some important properties have been studied.

\end{abstract}

\section{Introduction}
A semigroup is an algebraic structure consisting of a non-empty
set $S$ together with an associative binary operation\cite{H}. The
formal study of semigroups began in the early 20th century.
Semigroups are important in many areas of mathematics, for
example, coding and language theory, automata theory,
combinatorics and mathematical analysis. The concept of fuzzy sets
was introduced by \textit{Lofti Zadeh}\cite{Z} in his classic
paper in 1965. \textit{Azirel Rosenfeld}\cite{R} used the idea of
fuzzy set to introduce the notions of fuzzy subgroups.
\textit{Nobuaki Kuroki}\cite{K1,K2,K3} is the pioneer of fuzzy
ideal theory of semigroups. The idea of fuzzy subsemigroup was
also introduced by \textit{Kuroki}\cite{K1,K3}. In \cite{K2},
\textit{Kuroki} characterized several classes of semigroups in
terms of fuzzy left, fuzzy right and fuzzy bi-ideals. Others who
worked on fuzzy semigroup theory, such as \textit{X.Y.
Xie}\cite{X1,X2}, \textit{Y.B. Jun}\cite{J}, are mentioned in the
bibliography. \textit{X.Y. Xie}\cite{X1} introduced the idea of
extensions of fuzzy ideals in semigroups. \textit{K.H.
Kim}\cite{K} studied intuitionistic $Q$-fuzzy ideals in
semigroups. In this paper the notion of extension of $Q$-fuzzy
ideals in semigroups has been introduced and some important
properties have been investigated.
\section{Preliminaries}

In this section we discuss some elementary definitions that we use
in the sequel.\\

\begin{definition}
\cite{Mo} If $(S,\ast)$ is a mathematical system such that
$\forall a,b,c\in S,$ $(a\ast b)\ast c=a\ast(b\ast c),$ then
$\ast$ is called associative and $(S,\ast)$ is called a
\textit{semigroup}.
\end{definition}

\begin{definition}
\cite{Mo} A \textit{semigroup} $(S,\ast)$ is said to be
commutative if for all $a,b\in S,$ $a\ast b=b\ast a.$
\end{definition}

\begin{definition}
\cite{Mo} A semigroup $S$ is said to be \textit{left $($right$)$
regular} if, for each element $a$ of $S$, there exists an element
$x$ in $S$ such that $a=xa^{2}($resp. $a=a^{2}x).$
\end{definition}

\begin{definition}
\cite{Mo} A semigroup $S$ is called \textit{intra-regular} if for
each element $a$ of $S,$ there exist elements $x,y\in S$ such that
$a=xa^{2}y.$
\end{definition}

\begin{definition}
\cite{Mo} A semigroup $S$ is called \textit{regular} if for each
element $a$ of $S,$ there exists an element $x\in S$ such that
$a=axa.$
\end{definition}

\begin{definition}
\cite{Mo} A semigroup $S$ is called \textit{archimedean} if for
all $a,b\in S,$ there exists a positive integer $n$ such that
$a^{n}\in SbS.$
\end{definition}

Throughout the paper unless otherwise stated $S$ will denote a
semigroup.\\

\begin{definition}
\cite{Mo} A {\it subsemigroup} of a semigroup $S$ is a non-empty
subset $I$ of $S$ such that $I^{2} \subseteq I.$
\end{definition}

\begin{definition}
\cite{Mo} A subsemigroup $I$ of a semigroup $S$ is a called an
{\it interior ideal} of $S$ if $SIS\subseteq I.$
\end{definition}

\begin{definition}
\cite{Mo} A {\it left} ({\it right}){\it ideal} of a semigroup $S$
is a non-empty subset $I$ of $S$ such that $SI \subseteq I$ ($IS
\subseteq I$). If $I$ is both a left and a right ideal of a
semigroup $S$, then we say that $I$ is an {\it ideal} of $S$.
\end{definition}

\begin{definition}
\cite{Mo} Let $S$ be a semigroup. Then an ideal $I$ of $S$ is said
to be $(i)$ {\it completely prime} if $xy\in I$ implies that $x\in I$ or $y\in
I\forall x,y\in S,(ii)$ {\it completely semiprime} if $x^{2}\in I$ implies that
$x\in I,\forall x\in S$.
\end{definition}

\begin{definition}
\cite{Z} A {\it fuzzy subset} of a non-empty set $X$ is a function
$\mu:X\rightarrow [0,1].$
\end{definition}

\begin{definition}
\cite{S} Let $\mu$ be a fuzzy subset of a set $X,$ $\alpha
\in\lbrack 0,1-\sup \{\mu(x):x\in X\}].$ Then $\mu^{T}_{\alpha}$
is called a \textit{fuzzy translation} of $\mu$ if $\mu_{\alpha
}^{T}(x)=\mu(x)+\alpha\forall x\in X$.
\end{definition}

\begin{definition}
\cite{S} Let $\mu$ be a fuzzy subset of a set $X,$ $\beta\in
[0,1].$ Then $\mu^{M}_{\beta}$ is called a \textit{fuzzy
multiplication} of $\mu$ if $\mu_{\beta}^{M}(x)=\beta\cdot\mu(x)\forall x\in X$.
\end{definition}

\begin{definition}
\cite{S} Let $\mu$ be a fuzzy subset of a set $X,$ $\alpha
\in\lbrack 0,1-\sup \{\mu(x):x\in X\}],$ where $\beta\in [0,1].$
Then $\mu^{C}_{\beta\alpha}$ is called a \textit{fuzzy magnified
translation} of $\mu$ if $\mu_{\beta\alpha
}^{C}(x)=\beta\cdot\mu(x)+\alpha\forall x\in X$.
\end{definition}

\begin{definition}
Let $Q$ and $X$ be two non-empty sets. A mapping $\mu:X\times
Q\rightarrow [0,1]$ is called the {\it $Q$-fuzzy subset} of $X.$
\end{definition}

\begin{definition}
Let $\mu$ be a $Q$-fuzzy subset of a non-empty set $X.$ Then the
set $\mu_{t}=\{x\in X:\mu(x,q)\geq t\forall q\in Q\}$ for $t\in
[0,1],$ is called the {\it level subset} or {\it $t$-level subset} of $\mu.$
\end{definition}

\begin{example}
Let $S=\{a,b,c\}$ and $\ast$ be a binary operation on $S$ defined
in the following caley table:

\ \ \ \ \ \ \ \ \ \ \ \ \ \ \ \ \ \ \ \ \ \ \ \ \ \ \ \ \ \ \ \ \
\ \ \ \ \ $
\begin{tabular}{|c|c|c|c|}
\hline $\ast $ & $a$ & $b$ & $c$ \\ \hline $a$ & $a$ & $a$ & $a$
\\ \hline $b$ & $b$ & $b$ & $b$ \\ \hline $c$ & $c$ & $c$ & $c$ \\
\hline
\end{tabular}%
$\\

Then $S$ is a semigroup. Let $Q=\{p\}.$ Let us consider a
$Q$-fuzzy subset $\mu: S\times Q\rightarrow [0,1],$ by
$\mu(a,p)=0.8,\mu(b,p)=0.7,\mu(c,p)=0.6.$ For $t=0.7,$
$\mu_{t}=\{a,b\}.$
\end{example}

\section{$Q$-Fuzzy Ideals}

\begin{definition}
A non-empty $Q$-fuzzy subset of a semigroup $S$ is called a
\textit{$Q$-fuzzy subsemigroup} of $S$ if
$\mu(xy,q)\geq\min\{\mu(x,q),\mu(y,q)\}$ $\forall x,y\in S,\forall
q\in Q.$
\end{definition}

\begin{definition}
A $Q$-fuzzy subsemigroup $\mu$ of a semigroup $S$ is called a
\textit{$Q$-fuzzy interior ideal} of $S$ if
$\mu(xay,q)\geq\mu(a,q)\forall x,a,y\in S,\forall q\in Q.$
\end{definition}

\begin{definition}
A non-empty $Q$-fuzzy subset $\mu$ of a semigroup $S$ is called a
\textit{$Q$-fuzzy left$($right$)$ ideal} of $S$ if
$\mu(xy,q)\geq\mu(y,q)($resp. $\mu(xy,q)\geq\mu(x,q))$\ $\forall
x,y\in S,\forall q\in Q.$
\end{definition}

\begin{definition}
A non-empty $Q$-fuzzy subset $\mu$ of a semigroup $S$ is called a
\textit{$Q$-fuzzy two-sided ideal} or a \textit{$Q$-fuzzy ideal}
of $S$ if it is both a $Q$-fuzzy left and a $Q$-fuzzy right ideal
of $S.$
\end{definition}

\begin{definition}
A $Q$-fuzzy ideal $\mu$ of a semigroup $S$ is called a
\textit{$Q$-fuzzy completely prime ideal} of $S$ if
$\mu(xy,q)=\max\{\mu(x,q),\mu(y,q)\}\forall x,y\in S,\forall q\in
Q.$
\end{definition}

\begin{definition}
A $Q$-fuzzy ideal $\mu$ of a semigroup $S$ is called a
\textit{$Q$-fuzzy completely semiprime ideal} of $S$ if
$\mu(x,q)\geq\mu(x^{2},q)\forall x\in S,\forall q\in Q.$
\end{definition}

\begin{theorem}
Let $I$ be a non-empty subset of a semigroup $S$ and
$\chi_{I\times Q}$ be the characteristic function of $I\times Q,$
where $Q$ is any non-empty set. Then $I$ is a left ideal$($right
ideal, ideal, completely prime ideal, completely semiprime$)$ of $S$ if and only if
$\chi_{I\times Q}$ is a $Q$-fuzzy left ideal$($resp. $Q$-fuzzy
right ideal, $Q$-fuzzy ideal, $Q$-fuzzy completely prime ideal,
$Q$-fuzzy completely semiprime ideal$)$ of $S$.
\end{theorem}

\begin{proof}
Let $I$ be a left ideal of a semigroup $S.$ Let $x,y\in S,q\in Q$
, then $xy\in I$ if $y\in I.$ It follows that $\chi_{I\times
Q}(xy,q)=1=\chi_{I\times Q}(y,q).$ If $y\notin I,$ then
$\chi_{I\times Q}(y,q)=0.$ In this case $\chi_{I\times
Q}(xy,q)\geq0=\chi_{I\times Q}(y,q)$ $.$ Therefore $\chi_{I\times
Q}$ is a $Q$-fuzzy left ideal of $S.$

Conversely, let $\chi_{I\times Q}$ be a $Q$-fuzzy left ideal of
$S.$ Let $x,y\in I,q\in Q,$ then $\chi_{I\times
Q}(x,q)=\chi_{I\times Q}(y,q)=1.$ Now let $x\in I$ and $s\in
S,q\in Q.$ Then $\chi_{I\times Q}(x,q)=1.$ Also $\chi_{I\times
Q}(sx,q)\geq\chi_{I\times Q}(x,q)=1.$ Thus $sx\in I.$ So $I$ is a
left ideal of $S.$ Similarly we can prove that the other parts of
the theorem.
\end{proof}

\begin{theorem}
Let $S$ be a semigroup, $Q$ be any non-empty set and $\mu$ be a
non-empty $Q$-fuzzy subset of $S,$ then $\mu$ is a $Q$-fuzzy left
ideal$(Q$-fuzzy right ideal, $Q$-fuzzy ideal, $Q$-fuzzy completely
prime ideal, $Q$-fuzzy completely semiprime ideal$)$ of $S$ if and
only if $\mu_{t}$'s are left ideals$($resp. right ideals, ideals,
completely prime ideals, completely semiprime ideals$)$ of $S$ for all $t\in Im(\mu),$
where $\mu_{t}=\{x\in S:\mu(x,q)\geq t\forall q\in Q\}.$
\end{theorem}

\begin{proof}
Let $\mu$ be a $Q$-fuzzy left ideal of $S.$ Let
$t\in\operatorname{Im}\mu,$ then there exist some $\alpha\in S$
and $q\in Q$ such that $\mu(\alpha,q)=t$ and so
$\alpha\in\mu_{t.}$ Thus $\mu_{t}\neq\phi.$ Again let $s\in S$ and
$x\in\mu_{t}.$ Now $\mu(sx,q)\geq\mu(x,q)\geq t.$ Therefore
$sx\in\mu_{t}.$ Thus $\mu_{t}$ is a left ideal of $S.$

Conversely, let $\mu_{t}$'s are left ideals of $S$ for all
$t\in\operatorname{Im}\mu.$ Again let $x,s\in S$ be such that
$\mu(x,q)=t\forall q\in Q.$ Then $x\in\mu_{t}.$ Thus
$sx\in\mu_{t}$ $($since $\mu_{t}$ is a left ideal of $S).$
Therefore $\mu(sx,q)\geq t=\mu(x,q).$ Hence $\mu$ is a $Q$-fuzzy
left ideal of $S.$ Similarly we can prove other parts of the
theorem.
\end{proof}

\section{$Q$-Fuzzy Ideal Extensions}

\begin{definition}
Let $S$ be a semigroup, $\mu$ be a $Q$-fuzzy subset of $S$ and
$x\in S.$ The $Q$-fuzzy subset $<x,\mu>$ where $<x,\mu>:S\times
Q\rightarrow [0,1]$ is defined by $<x,\mu>(y,q):=\mu(xy,q)\forall
y\in S,\forall q\in Q$ is called the \textit{$Q$-fuzzy extension}
of $\mu$ by $x.$
\end{definition}

\begin{example} Let $X=\{1,\omega,\omega^{2}\}$ and $Q=\{p\}.$ Let $\mu$ be a $Q$-fuzzy subset of $X,$ defined
as follows%
\begin{align*}
\mu(x,p)=\left\{
\begin{array}{ll}
0.3 & \text{if} \ x=1 \\
0.1 & \text{if} \ x=\omega \\
0.5 & \text{if} \ x=\omega^{2}
\end{array}
\right..
\end{align*}

Let $x=\omega.$ Then the $Q$-fuzzy extension of $\mu$ by $\omega$ is given by
\begin{align*}
<x,\mu>(y,p)=\left\{
\begin{array}{ll}
0.1 & \text{if} \ y=1 \\
0.5 & \text{if} \ y=\omega \\
0.3 & \text{if} \ y=\omega^{2}
\end{array}
\right..
\end{align*}
\end{example}

\begin{proposition}
Let $S$ be a commutative semigroup, $Q$ be any non-empty set and
$\mu$ be a $Q$-fuzzy ideal of $S.$ Then $<x,\mu>$ is a $Q$-fuzzy
ideal of $S\forall x\in S.$
\end{proposition}

\begin{proof}
Let $\mu$ be a $Q$-fuzzy ideal of a commutative semigroup $S$ and $y,z\in S,q\in Q.$ Then%
\begin{align*}
<x,\mu>(yz,q) & =\mu(xyz,q)\geq\mu(xy,q) =<x,\mu>(y,q)
\end{align*}

Hence $<x,\mu>$ is a $Q$-fuzzy right ideal of $S.$ Hence $S$ being
commutative, $<x,\mu>$ is a $Q$-fuzzy ideal of $S.$
\end{proof}

\begin{remark}
Commutativity of the semigroup $S$ is not required to prove that
$<x,\mu>$ is a $Q$-fuzzy right ideal of $S$ when $\mu$ is a
$Q$-fuzzy right ideal of $S.$
\end{remark}

\begin{definition}
Let $S$ be a semigroup, $Q$ be any non-empty set and $\mu$ be a $Q$-fuzzy subset of $S.$
Then we define $Supp$ $ \mu=\{x\in S:\mu(x,q)>0\forall q\in Q\}.$
\end{definition}

\begin{proposition}
Let $S$ be a semigroup, $Q$ be any non-empty set, $\mu$ be a $Q$-fuzzy ideal of $S$ and
$x\in S.$ Then we have the following:

$(i)$ $\mu\subseteq <x,\mu>.$

$(ii)$ $<x^{n},\mu>\subseteq <x^{n+1},\mu>\forall n\in N.$

$(iii)$ If $\mu(x,q)>0$ then $Supp <x,\mu>=S\times Q.$
\end{proposition}

\begin{proof}
$(i)$ Let $\mu$ be a $Q$-fuzzy ideal of $S$ and let $y\in S,q\in
Q.$ Then $<x,\mu>(y,q)=\mu(xy,q)\geq\mu(y,q).$ Hence $\mu\subseteq
<x,\mu>.$

$(ii)$ Let $\mu$ be a $Q$-fuzzy ideal of $S$ and $n\in N,y\in
S,q\in Q.$ Then
$<x^{n+1},\mu>(y,q)=\mu(x^{n+1}y,q)=\mu(xx^{n}y,q)\geq\mu(x^{n}y,q)=<x^{n},\mu>(y,q).$
Hence $<x^{n},\mu>\subseteq <x^{n+1},\mu>\forall n\in N.$

$(iii)$ Since $<x,\mu>$ is a $Q$-fuzzy subset of $S,$ so by
definition, $Supp$ $<x,\mu>\subseteq S\times Q.$ Let $\mu$ be a
$Q$-fuzzy ideal of $S$ and $y\in S,q\in Q.$ Then
$<x,\mu>(y,q)=\mu(xy,q)\geq\mu(x,q)> 0.$ Consequently, $(y,q)\in
Supp$ $<x,\mu>.$ Thus $S\times Q\subseteq Supp$ $<x,\mu>.$ Hence
$Supp<x,\mu>=S\times Q.$
\end{proof}

\begin{definition}
Let $S$ be a semigroup, $Q$ be any non-empty set, $A\subseteq S$
and $x\in S.$ We define $<x,A\times Q>=\{y\in S,q\in Q:(xy,q)\in
A\times Q\}.$
\end{definition}

\begin{proposition}
Let $S$ be a semigroup, $Q$ be any non-empty set and $\phi\neq
A\subseteq S.$ Then $<x,\mu _{A\times Q}>=\mu_{<x.A\times Q>}$ for
every $x\in S,$ where $\mu_{A\times Q}$ denotes the characteristic
function of $A\times Q.$
\end{proposition}

\begin{proof}
Let $\ x,y\in S,q\in Q.$  Then two cases may arise viz. Case
$(i):$ $(y,q)\in<x,A\times Q>.$ Case $(ii):$
$(y,q)\notin<x,A\times Q>$.

Case $(i):$ $(y,q)\in<x,A\times Q>.$ Then $(xy,q)\in A\times Q.$
This means that $\mu_{A\times Q}(xy,q)=1$ whence $<x,\mu_{A\times
Q}>(y,q)=1.$ Also $\mu_{<x,A\times Q>}(y,q)=1.$

Case $(ii):$ $(y,q)\notin<x,A\times Q>$. Then $(xy,q)\notin
A\times Q$. So $\mu_{A\times Q}(xy,q)=0.$ Thus $<x,\mu_{A\times
Q}>(y,q)=0.$ Again $\mu_{<x,A\times Q>}(y,q)=0.$ Hence we conclude
$<x,\mu_{A\times Q}>=\mu_{<x,A\times Q>}.$
\end{proof}

\begin{proposition}
Let $S$ be a commutative semigroup, $Q$ be any non-empty set and
$\mu$ be a $Q$-fuzzy completely prime ideal of $S.$ Then $<x,\mu>$
is a $Q$-fuzzy completely prime ideal of $S\forall x\in S.$
\end{proposition}

\begin{proof}
Let $\mu$ be a $Q$-fuzzy completely prime ideal of $S.$ Then by Proposition $4.2,$ $<x,\mu>$ is a $Q$-fuzzy ideal of $S.$ Let $y,z\in S,q\in Q.$ Then%

\begin{align*}
<x,\mu>(yz,q) &= \mu(xyz,q)(cf. \text{ Definition } 4.1)=\max\{\mu(x,q),\mu(yz,q)\}(cf.\text{ Definition }\\ & 3.5)=\max\{\mu(x,q),\max\{\mu(y,q),\mu(z,q)\}\}\\
& =\max\{\max\{\mu(x,q),\mu(y,q)\},\max\{\mu(x,q),\mu(z,q)\}\}\\
& =\max\{\mu(xy,q),\mu(xz,q)\}(cf.\text{ Definition } 3.5)\\
& =\max\{<x,\mu>(y,q),<x,\mu>(z,q)\}(cf.\text{ Definition } 4.1)
\end{align*}

Hence $<x,\mu>$ is a $Q$-fuzzy completely prime ideal of $S.$
\end{proof}

\begin{remark}
Let $S$ be a semigroup, $Q$ be any non-empty set and $\mu$ be a
$Q$-fuzzy completely prime ideal of $S.$ Then
$<x,\mu>=<x^{2},\mu>.$
\end{remark}

\begin{proposition}
Let $S$ be a semigroup, $Q$ be any non-empty set and $\mu$ be a
non-empty $Q$-fuzzy subset of $S.$ Then for any $t\in
[0,1],<x,\mu_{t}>=<x,\mu>_{t}$ $\forall x\in S.$
\end{proposition}

\begin{proof}
Let $y\in <x,\mu>_{t},q\in Q.$ Then $<x,\mu>(y,q)\geq t.$ This
gives $\mu(xy,q)\geq t$ and hence $xy\in\mu_{t}.$ Consequently,
$y\in <x,\mu_{t}>.$ It follows that $<x,\mu>_{t}\subseteq
<x,\mu_{t}>.$ Reversing the above argument we can deduce that
$<x,\mu_{t}>\subseteq <x,\mu>_{t}.$ Hence
$<x,\mu>_{t}=<x,\mu_{t}>.$
\end{proof}

\begin{proposition}
Let $S$ be a commutative semigroup, $Q$ be any non-empty set and
$\mu$ be a $Q$-fuzzy subset of $S$ such that $<x,\mu>=\mu$ for
every $x\in S.$ Then $\mu$ is a constant function.
\end{proposition}

\begin{proof}
Let $x,y\in S, q\in Q.$ Then by hypothesis we have
$\mu(x,q)=<y,\mu>(x,q)=\mu(yx,q)=\mu(xy,q)($since $S$ is
commutative$)=<x,\mu>(y,q)=\mu(y,q).$ Hence $\mu$ is a constant
function.
\end{proof}

\begin{corollary}
Let $S$ be a commutative semigroup, $Q$ be any non-empty set and
$\mu$ be a $Q$-fuzzy completely prime ideal of $S.$ If $\mu$ is
not constant, $\mu$ is not a maximal $Q$-fuzzy completely prime
ideal of $S.$
\end{corollary}

\begin{proof}
Let $\mu$ be a $Q$-fuzzy completely prime ideal of $S.$ Then, by
Proposition $4.7,$ for each $x\in S,<x,\mu>$ is a $Q$-fuzzy
completely prime ideal of $S.$ Now by Proposition $4.4(i),$
$\mu\subseteq<x,\mu>$ for all $x\in S.$ If $\mu=<x,\mu>$ for all
$x\in S$ then by Proposition $4.9,$ $\mu$ is constant which is not
the case by hypothesis. Hence there exists $x\in S$ such that
$\mu\subsetneq<x,\mu>.$ This completes the proof.
\end{proof}

\begin{proposition}
Let $S$ be a commutative semigroup, $Q$ be any non-empty set and
$\mu$ be a $Q$-fuzzy completely semiprime ideal of $S.$ Then
$<x,\mu>$ is a $Q$-fuzzy completely semiprime ideal of $S$ for all
$x\in S.$
\end{proposition}

\begin{proof}
Let $\mu$ be a $Q$-fuzzy completely semiprime ideal of $S$ and $q\in Q.$ Then $\mu$ is a $Q$-fuzzy ideal of $S$ and hence by Proposition $4.2,$ $<x,\mu>$ is a $Q$-fuzzy ideal of $S.$ Let $y\in S.$ Then%

\begin{align*}
<x,\mu>(y^{2},q) &= \mu(xy^{2},q)(cf.\text{ Definition } 4.1)\leq\mu(xy^{2}x,q)(\text{ since } \mu \text{ is a Q-fuzzy ideal}\\
& \text{ ideal of } S)=\mu(xyyx,q)=\mu(xyxy,q)(\text{ since }S\text{ is commutative })\\
& =\mu((xy)^{2},q)\leq\mu(xy,q)(cf.\text{ Definition } 3.6)\\
& =<x,\mu>(y,q)(cf.\text{ Definition } 4.1)
\end{align*}

Hence $<x,\mu>$ is a $Q$-fuzzy completely semiprime ideal of $S.$
\end{proof}

\begin{corollary}
Let $S$ be a commutative semigroup, $Q$ be any non-empty set and
$\{\mu_{i}\}_{i\in \Lambda}$ be a family of $Q$-fuzzy completely
semiprime ideals of $S.$ Let $\lambda=\underset{i\in
\Lambda}{\bigcap}\mu_{i}.$ Then for any $x\in S,$ $<x,\lambda>$ is
a $Q$-fuzzy completely semiprime ideal of $S,$ provided $\lambda$
is non-empty.
\end{corollary}

\begin{proof}
Let $x,y\in S,q\in Q.$ Then%
\begin{align*}
\lambda(xy,q) & =\underset{i\in
\Lambda}{\inf}\mu_{i}(xy,q)\geq\underset{i\in
\Lambda}{\inf}\mu_{i}(x,q)=\lambda(x,q)
\end{align*}

Hence $S$ being commutative semigroup, $\lambda$ is a $Q$-fuzzy
ideal of $S.$

Again let $a\in S,q\in Q.$ Then%
\begin{align*}
\lambda(a,q) & =\underset{i\in
\Lambda}{\inf}\mu_{i}(a,q)\geq\underset{i\in
\Lambda}{\inf}\mu_{i}(a^{2},q)=\lambda(a^{2},q)
\end{align*}

Consequently, $\lambda=\underset{i\in \Lambda}{\bigcap}\mu_{i}$ is
a $Q$-fuzzy completely semiprime ideal of $S.$ Hence by
Proposition $4.11,$ $<x,\lambda>$ is a $Q$-fuzzy completely
semiprime ideal of $S.$
\end{proof}

\begin{remark}
The proof of the above corollary shows that in a semigroup the
non-empty intersection of family of $Q$-fuzzy completely semiprime
ideals is a $Q$-fuzzy completely semiprime ideal.
\end{remark}

\begin{corollary}
Let $S$ be a commutative $\Gamma$-semigroup, $Q$ be any non-empty
set and $\{S_{i}\}_{i\in I}$ a non-empty family of completely semiprime
ideals of $S$ and $A:=\cap_{i\in I}S_{i}\neq\phi.$ Then
$<x,\mu_{A}>$ is a $Q$-fuzzy completely semiprime ideal of $S$ for
all $x\in S,$ where $\mu_{A}$ is the characteristic function of
$A.$
\end{corollary}

\begin{proof}
$A=\cap_{i\in I}S_{i}\neq\phi($by the given condition$).$ Hence
$\mu_{A\times Q}\neq\phi.$ Let $x\in S,q\in Q.$ Then $x\in A$ or
$x\notin A.$ If $(x,q)\in A\times Q$ then $\mu_{A\times Q}(x,q)=1$ and $(x,q)\in S_{i}\times Q$ $\forall i\in I.$ Hence%
\[
\inf\{\mu_{S_{i}\times Q}:i\in I\}(x,q)=\underset{i\in I}{\inf}\{\mu_{S_{i}\times Q}%
(x,q)\}=1=\mu_{A\times Q}(x,q).
\]
If $x\notin A$ then $\mu_{A\times Q}(x,q)=0$ and for some $i\in
I,$ $(x,q)\notin S_{i}\times Q.$ It
follows that $\mu_{S_{i}\times Q}(x,q)=0.$ Hence%
\[
\inf\{\mu_{S_{i}\times Q}:i\in I\}(x,q)=\underset{i\in I}{\inf}\{\mu_{S_{i}\times Q}%
(x,q)\}=0=\mu_{A\times Q}(x,q).
\]
Thus we see that $\mu_{A\times Q}=\inf\{\mu_{S_{i}\times Q}:i\in
I\}.$ Hence $\mu_{A}=\inf\{\mu_{S_{i}}:i\in I\}(cf.$ Corollary $4.12).$ Again
$\mu_{S_{i}}$ is a $Q$-fuzzy completely semiprime ideal of $S$ for
all $i\in I.$ Consequently by Corollary
$4.13,$ for all $x\in S,<x,\mu_{A}>$ is a $Q$-fuzzy completely
semiprime ideal of $S.$
\end{proof}

We can obtain following results by routine verification.

\begin{theorem}
Let $S$ be a semigroup, $Q$ be any non-empty set and $\mu$ be a
$Q$-fuzzy completely prime ideal of $S.$ Then
$<x,\mu_{\beta\alpha}^{C}>$ is a $Q$-fuzzy completely prime ideal
of $S.$
\end{theorem}

\begin{theorem}
Let $S$ be a semigroup, $Q$ be any non-empty set and $\mu$ be a
$Q$-fuzzy right ideal of $S.$ Then $<x,\mu_{\beta\alpha}^{C}>$ is
a $Q$-fuzzy right ideal of $S.$
\end{theorem}

\begin{theorem}
Let $S$ be a commutative semigroup, $Q$ be any non-empty set and
$\mu$ be a $Q$-fuzzy ideal of $S.$ Then
$<x,\mu_{\beta\alpha}^{C}>$ is a $Q$-fuzzy ideal of $S.$
\end{theorem}

\begin{theorem}
Let $S$ be a commutative semigroup and $Q$ be any non-empty set
and $\mu$ be a $Q$-fuzzy completely semiprime ideal of $S.$ Then
$<x,\mu_{\beta\alpha}^{C}>$ is a $Q$-fuzzy completely semiprime
ideal of $S.$
\end{theorem}

\begin{theorem}
Let $S$ be a commutative semigroup and $Q$ be any non-empty set
and $\mu$ be a $Q$-fuzzy interior ideal of $S.$ Then
$<x,\mu_{\beta\alpha}^{C}>$ is a $Q$-fuzzy interior ideal of $S.$
\end{theorem}

\begin{theorem}
Let $S$ be a regular commutative semigroup, $Q$ be any non-empty
set and $\mu$ be a $Q$-fuzzy ideal of $S.$ Then
$<x,\mu_{\beta\alpha}^{C}>$ is a $Q$-fuzzy completely semiprime
ideal of $S.$
\end{theorem}

\begin{theorem}
Let $S$ be a right regular semigroup, $Q$ be any non-empty set and
$\mu$ be a $Q$-fuzzy right ideal of $S.$ Then
$<x,\mu_{\beta\alpha}^{C}>$ is a $Q$-fuzzy completely semiprime
right ideal of $S.$
\end{theorem}

\begin{theorem}
Let $S$ be an intra-regular commutative semigroup, $Q$ be any
non-empty set and $\mu$ be a $Q$-fuzzy ideal of $S$. Then
$<x,\mu_{\beta\alpha}^{C}>$ is a $Q$-fuzzy completely semiprime
ideal of $S.$
\end{theorem}

\begin{theorem}
Let $S$ be an archimedean commutative semigroup, $Q$ be any
non-empty set and $\mu$  be a $Q$-fuzzy completely semiprime ideal
$<x,\mu_{\beta\alpha}^{C}>$ of $S$ is a constant function.
\end{theorem}

\begin{remark}
If we put $\beta=1($respectively $\alpha=0)$ in fuzzy magnified
translation then it reduces to fuzzy translation$($respectively
fuzzy multiplication$).$ Consequently analogues of Theorems
$4.14$-$4.22$ follow easily in fuzzy translation and fuzzy
multiplication.
\end{remark}



\begin{thebibliography}{9}

\bibitem{H} J. Howie; {\it  Fundamentals of semigroup theory}, London Mathematical Society Monographs. New Series, 12. Oxford Science Publications. The Clarendon Press, Oxford University Press, New York, 1995.\vspace{-.1 in}

\bibitem {J}Y.B. Jun, S.M. Hong and J. Meng; {\it Fuzzy interior ideals in semigroups,} Indian J. of Pure Appl. Math., 26(9)(1995) 859-863. \vspace{-.1 in}

\bibitem {K}K.H. Kim; {\it On intuitionistic $Q$-fuzzy semiprime ideals in semigroups,} Advances in Fuzzy Mathematics, 1(1) (2006) 15-21. \vspace{-.1 in}

\bibitem {K1}N. Kuroki; {\it On fuzzy ideals and fuzzy bi-ideals in semigroups,} Fuzzy Sets and Systems, 5(1981) 203-215. \vspace{-.1 in}

\bibitem {K2}N. Kuroki; {\it On fuzzy semigroups,} Inform. Sci., 53(1991) 203-236. \vspace{-.1 in}

\bibitem {K3}N, Kuroki; {\it Fuzzy semiprime quasi ideals in semigroups,} Inform. Sci., 75(3)(1993) 201-211. \vspace{-.1 in}

\bibitem {Mo}Mordeson et all; {\it Fuzzy semigroups,} Springer-Verlag, (2003), Heidelberg.

\bibitem{R} A. Rosenfeld; {\it Fuzzy groups,} J. Math. Anal. Appl., 35 (1971) 512-517. \vspace{-.1 in}

\bibitem{S} S.K. Sardar, M. Mandal and S.K. Majumder; {\it On intuitionistic fuzzy magnified translation in semigroups,}(Communicated).\vspace{-.1 in}

\bibitem {X1}Xiang-Yun Xie; {\it Fuzzy ideal extensions of semigroups,} Soochow Journal of Mathematics, 27(2)(April 2001) 125-138. \vspace{-.1 in}

\bibitem {X2}Xiang-Yun Xie; {\it Fuzzy ideal extensions of ordered semigroups,} Lobach Journal of Mathematics, 19(2005) 29-40. \vspace{-.1 in}

\bibitem{Z} L.A. Zadeh; {\it Fuzzy sets,} Information and Control, 8 (1965) 338-353.\vspace{-.1 in}
\end{thebibliography}
\end{document}